\documentclass[a4paper]{article}
\usepackage{amsthm}
\usepackage{amsmath}
\usepackage{amsfonts}
\usepackage{amssymb,a4wide}
\usepackage{graphicx}

\newtheorem{theorem}{\textbf{Theorem}}
\newtheorem{lemma}[theorem]{\textbf{Lemma}}

\newtheorem{definition}{\textbf{Definition}}

\newtheorem{remark}[theorem]{\textbf{Remark}}
\newtheorem{problem}{\textbf{Problem}}

\newcommand{\btree}{$\beta(1,0)$\mbox{-tree}}

\DeclareMathOperator{\roots}{\mathrm{root}}
\DeclareMathOperator{\rsub}{\mathrm{rsub}}
\DeclareMathOperator{\sub}{\mathrm{sub}}
\DeclareMathOperator{\rpath}{\mathrm{rpath}}

\title{Enumeration of fixed points of an involution on $\beta(1,0)$-trees}
\author{Sergey Kitaev\footnote{Department of Computer and Information Sciences, University of Strathclyde, Glasgow, United Kingdom. Email: sergey.kitaev@cis.strath.ac.uk}\ \ and Anna de Mier\footnote{ Departament de Matem\`atica Aplicada 2, Universitat Polit\`{e}cnica de Catalunya, Barcelona, Spain. Email: anna.de.mier@upc.edu.}}

\begin{document}
\maketitle

\begin{abstract} $\beta(1,0)$-trees provide a convenient description of rooted non-separable planar maps. The involution $h$ on $\beta(1,0)$-trees was introduced to prove a complicated equidistribution result on a class of pattern-avoiding permutations. In this paper, we describe and enumerate fixed points of the involution $h$. Intriguingly, the fixed points are equinumerous with the fixed points under taking the dual map on rooted non-separable planar maps, even though the fixed points do not go to each other under the know (natural) bijection between the trees and the maps.\end{abstract}

\section{Introduction}

A special case of so-called {\em description trees} introduced in \cite{CJS} to describe several classes of {\em planar maps} can be defined as follows.

\begin{definition} A \emph{\btree} is a rooted
plane tree labeled with positive integers such that
\begin{enumerate}
\item Leaves have label $1$.
\item The root has label equal to the sum of its children's labels.
\item Any other node has an integer label between $1$ and the sum of its
  children's labels.
\end{enumerate}
\end{definition}

For example, all $\beta(1,0)$-trees on 3 edges are presented in Figure~\ref{beta10}.

\begin{figure}[h]
\begin{center}
\includegraphics[scale=0.5]{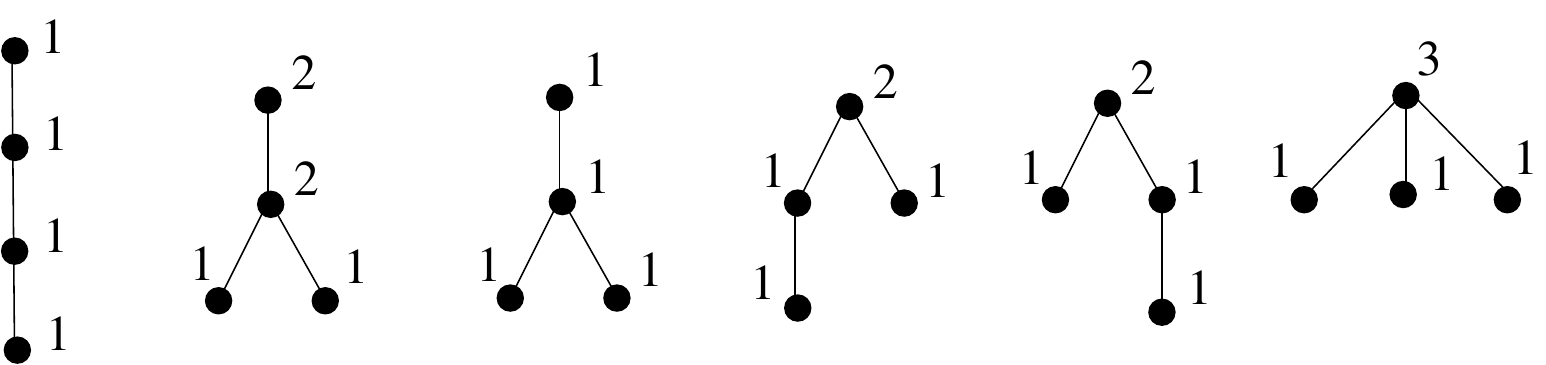}
\caption{All $\beta(1,0)$-trees on 4 nodes.}\label{beta10}
\end{center}
\end{figure}

It turns out that $\beta(1,0)$-trees are in one-to-one correspondence with {\em rooted non-separable planar maps} (see Section \ref{sec-RNPM} for details), thus providing a useful tool to work with the maps (see \cite{CS,KSSU}). Moreover, they are also in bijection with 2-{\em stack sortable permutations} and permutations avoiding simultaneously the patterns 3142 and 2\underline{41}3 (see \cite{Kit} for a comprehensive overview over the field of permutation patterns and for definitions of the mentioned objects).

The involution $h$ on $\beta(1,0)$-trees (to be reviewed in Section \ref{sec2}) was introduced in \cite{CKS1} in order to prove a complicated equidistribution result on (3142,2\underline{41}3)-avoiding permutations. A natural question to  ask is what can be said about the fixed points of $h$. Can we describe their structure? Can we enumerate them? Can we link them to other (combinatorial) objects?

In this paper we will address these questions. The paper is organized as follows. In Section \ref{sec2} we not only define the involution $h$, but also provide a sketch of a proof (originally appearing in \cite{CKS2}) that this map is indeed an involution. The structure of the fixed points is then described in Section  \ref{sec-Structure}, and they are enumerated in Section \ref{sec-Enum}. Section \ref{sec-RNPM} links our studies to the fixed points under the duality map  on rooted non-separable planar maps studied in \cite{KMN}; three open problems are raised in that section. More (bijective) open problems can be found in Section \ref{more}. 

\section{The involution $h$ on $\beta(1,0)$-trees}\label{sec2}

To proceed, we need to define several statistics on $\beta(1,0)$-trees. These are given in Table~\ref{stats10trees}. For example, for the $\beta(1,0)$-tree $T$ in Figure~\ref{beta10tree}, the values of the statistics are as follows: $\roots(T)=5$, $\sub(T)=3$,
$\rpath(T)=3$, and $\rsub(T)=2$. For another example, the second tree from left to right in Figure~\ref{beta10} has
$\roots(T)=2$, $\sub(T)=1$,
$\rpath(T)=2$, and $\rsub(T)=1$.

\begin{table}
\begin{center}
\begin{tabular}{l|l}
Statistic & Description in a $\beta(1,0)$-tree $T$\\
\hline
$\roots(T)$ & root's label\\
\hline
$\sub(T)$ & \# children of the root\\ & = \# subtrees coming out from the root \\
\hline
$\rpath(T)$ & \# edges from the root to the rightmost leaf \\
& = length of the rightmost path (right-path)\\
\hline
$\rsub(T)$ & \# 1s below the root on the right-path
\end{tabular}
\caption{Statistics on $\beta(1,0)$-trees as described in~\cite{CKS1}.}\label{stats10trees}
\end{center}
\end{table}

\begin{figure}[h]
\begin{center}
\includegraphics[scale=0.5]{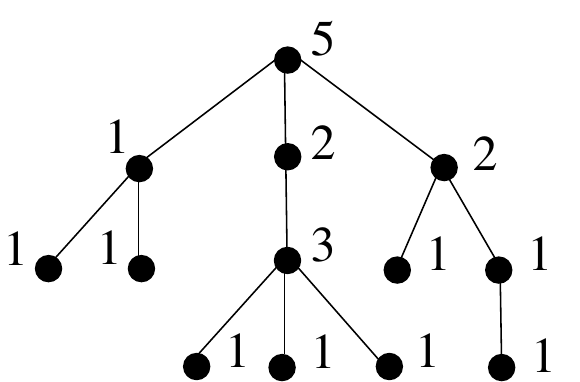}
\caption{A $\beta(1,0)$-tree.}\label{beta10tree}
\end{center}
\end{figure}

\begin{definition}
A $\beta(1,0)$-tree $T$ on at least two nodes is {\em indecomposable} if $\sub(T)=1$, that is, if the root of $T$ has exactly one child; otherwise, $T$ is {\em decomposable}.  A $\beta(1,0)$-tree $T$ on at least two nodes is {\em right-indecomposable} if $\rsub(T)=1$, that is, if the right-path has exactly one $1$ below the root; otherwise, $T$ is {\em right-decomposable}.
\end{definition}

The idea of the involution $h$ on $\beta(1,0)$-trees, defined in~\cite{CKS1}, is to turn $\beta(1,0)$-tree decompositions into right-decompositions, and vice versa. We define $h$ recursively (see a schematic description in Figure~\ref{involutionH}).
As the base case, we map the single node tree and the one edge tree to themselves. We also assume inductively that if $\mathrm{root}(A)=x$ then $\mathrm{rpath}(h(A))=x$ (except if $A$ has only one node).  In the case of an indecomposable tree, we remove the top edge to get the $\beta(1,0)$-tree $A$ (whose root may need to be adjusted), apply $h$ recursively to get $h(A)$,  add a new leaf to $h(A)$ so that the statistic $\mathrm{rpath}$ of the result equals $\mathrm{root}(A)$, and finally increase all  labels above this new rightmost leaf by 1. On the other hand, if the tree is decomposable, let $A$ be the tree induced by the root and all its subtrees but the rightmost one, and let $B$ be the tree induced by the root and its rightmost subtree  (again, adjusting the root labels if necessary). Then identify the rightmost leaf of $h(B)$ with the root of $h(A)$, this identified node keeping the  label $1$ of the leaf. 
 See Figure~\ref{invH} for an example of applying the involution $h$ together with some of the steps involved in the recursive procedure.

Actually, it is not only the case that $\mathrm{root}(A)=\mathrm{rpath}(h(A))$, but, as shown in \cite{CKS1},  under $h$ one can control 8 (mostly natural) statistics on $\beta(1,0)$-trees. However, for us it is enough to consider four of these statistics, which are mentioned in the following theorem.

\begin{theorem}[\cite{CKS1}]\label{niceProperty} If $S=h(T)$ then $\roots(T)=\rpath(S)$, $\roots(S)=\rpath(T)$, $\sub(T)=\rsub(S)$ and $\sub(S)=\rsub(T)$.\end{theorem}

\begin{figure}[h]
\begin{center}
\includegraphics[scale=0.45]{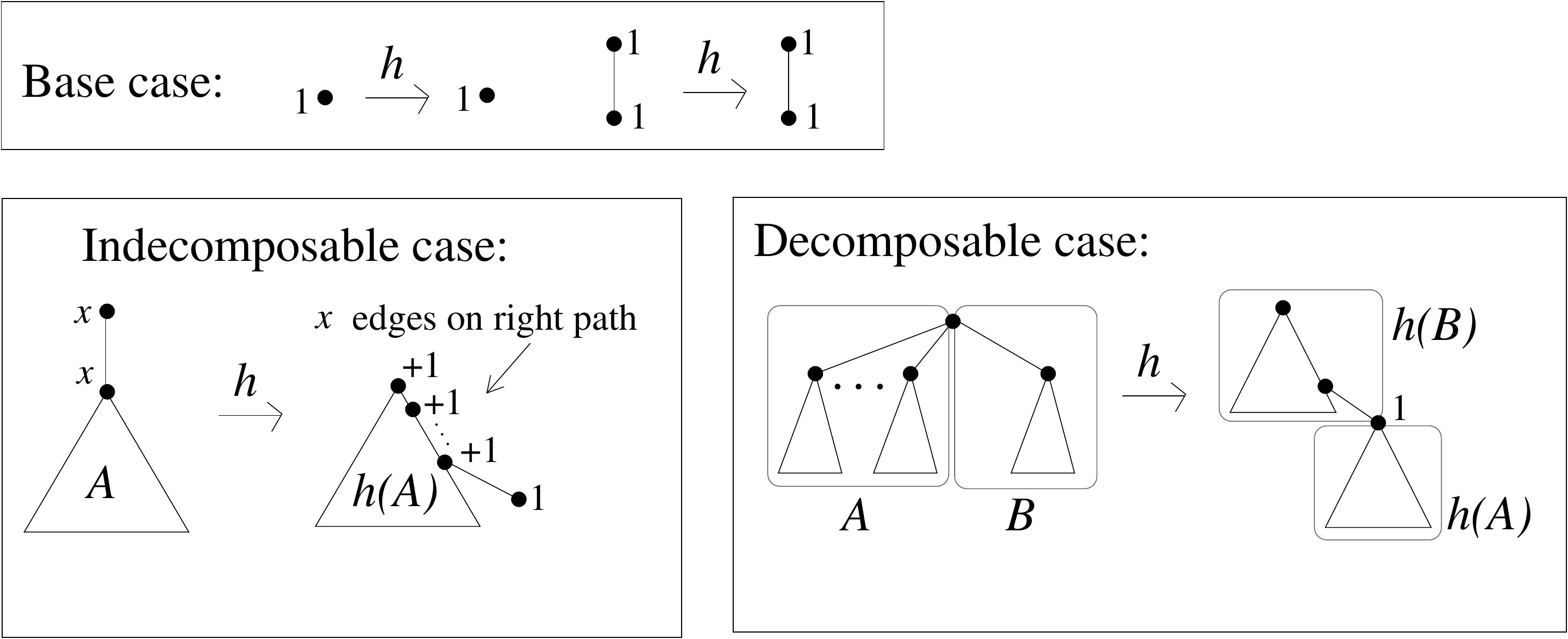}
\caption{A schematic description of the involution $h$. In the indecomposable case, a ``$+1$'' next to a node has to be interpreted as increasing the label of that node by $1$.}\label{involutionH}
\end{center}
\end{figure}

\begin{figure}[h]
\begin{center}
\includegraphics[scale=0.45]{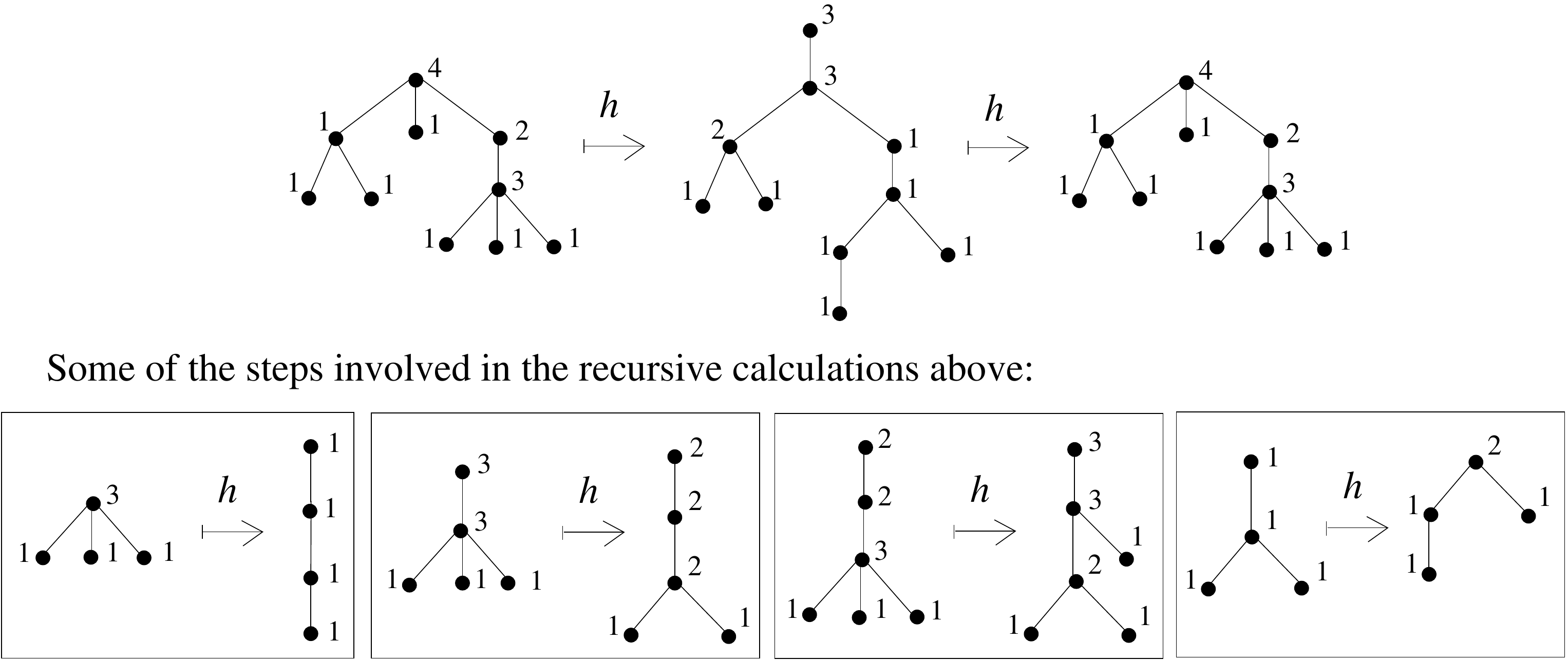}
\caption{An example of applying the involution $h$ together with some of the steps involved in the recursive procedure.}\label{invH}
\end{center}
\end{figure}

It was mentioned in \cite{CKS1} that  specializing $h$ on those $\beta(1,0)$-trees whose non-root nodes are all labelled 1 gives an interesting involution on structures counted by {\em Catalan numbers}, which provides an extra motivation to study $h$. Even though $h$ was defined and used in \cite{CKS1}, a proof that $h$ is actually an involution has not appeared until it was presented in a formal way in \cite{CKS2}. In either case, below we provide a sketch of the proof that involves a picture; this gives an intuitive idea on the non-formal proof that the authors of \cite{CKS1} originally came up with, and it is helpful for better understanding the structure of fixed points we provide below. 

Let us before make some comments on the  figures that appear in the proof that $h$ is an involution in Theorem \ref{h-is-involution} below and in rest of the paper.  Whenever a subtree is labelled by a capital letter, say $A$, it has to be understood that $A$ is the $\beta(1,0)$-tree with the same nodes and edges as the subtree, and with the same labelling, except possibly for the root node. Similarly, a subtree labelled $h(A)$ agrees with the image of $A$ under $h$, except perhaps at the root label. As in Figure~\ref{involutionH}, a ``$+1$'' next to a node means that the original label of that node goes up by $1$, and a ``$>1$'' means that the label of that node is greater than $1$. Also, when several nodes have the mark ``$>1$'' along a path, it has to be understood as a (possibly empty) sequence  of nodes with labels greater than~$1$.

\begin{theorem}\label{h-is-involution}(\cite{CKS2}) The map $h$ is an involution, that is, $h^2(T)=T$. \end{theorem}

\begin{proof} 

As mentioned above, we provide a sketch of a proof. The proof is based on induction on the size of a tree with the obvious base case, the one node tree going to itself.  

We would like to prove first a property of $h$ that is shown schematically in Figure \ref{case-to-prove}, where a given $\beta(1,0)$-tree $T$ is represented using the right-decomposition, while $h(T)$ is represented using (the usual) decomposition. Observe that the case $k=1$ follows readily from the definition of $h$ (recall Figure~\ref{involutionH}). Once this property is proved, we can apply $h$ to $h(T)$ in Figure \ref{case-to-prove} using the definition of $h$, and then apply the induction hypothesis (saying that $h^2(T)=T$ for smaller trees) to get the original tree $T$. The property in Figure \ref{case-to-prove} is also to be proved by (parallel) induction on the size of the tree.

\begin{figure}[h]
\begin{center}
\includegraphics[scale=0.4]{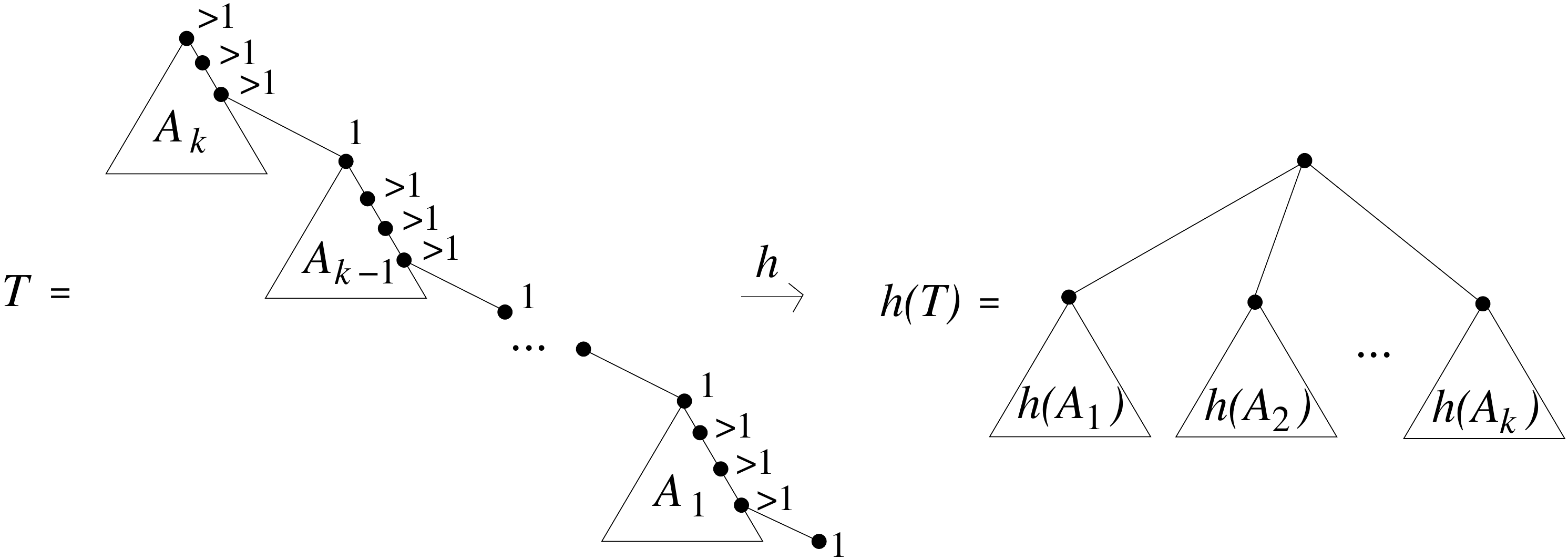}
\caption{A desired property of $h$ implying immediately that $h$ is an involution.}\label{case-to-prove}
\end{center}
\end{figure}

To prove the property in Figure \ref{case-to-prove}, we begin with decomposing the topmost tree $A_k$, as shown by the first equality in Figure \ref{case-proof} (the figure shows the case where $S_m$ is non-empty; the case it is empty follows in the same way). We can then use the definition of $h$, resulting in the second equality in  Figure \ref{case-proof}. Now we can apply our induction hypothesis on the property in Figure \ref{case-to-prove} to the smaller tree with right-decomposition based on the trees $S_m$, $A_{k-1}$, $A_{k-2}$, $\ldots$, $A_1$, to obtain the $\beta(1,0)$-tree on the left of Figure \ref{case-proof-2}. Since the rightmost subtree of the root in this tree is nothing else but $h(A_k)$, this gives us the desired result.

\begin{figure}[h]

\begin{center}
\includegraphics[scale=0.4]{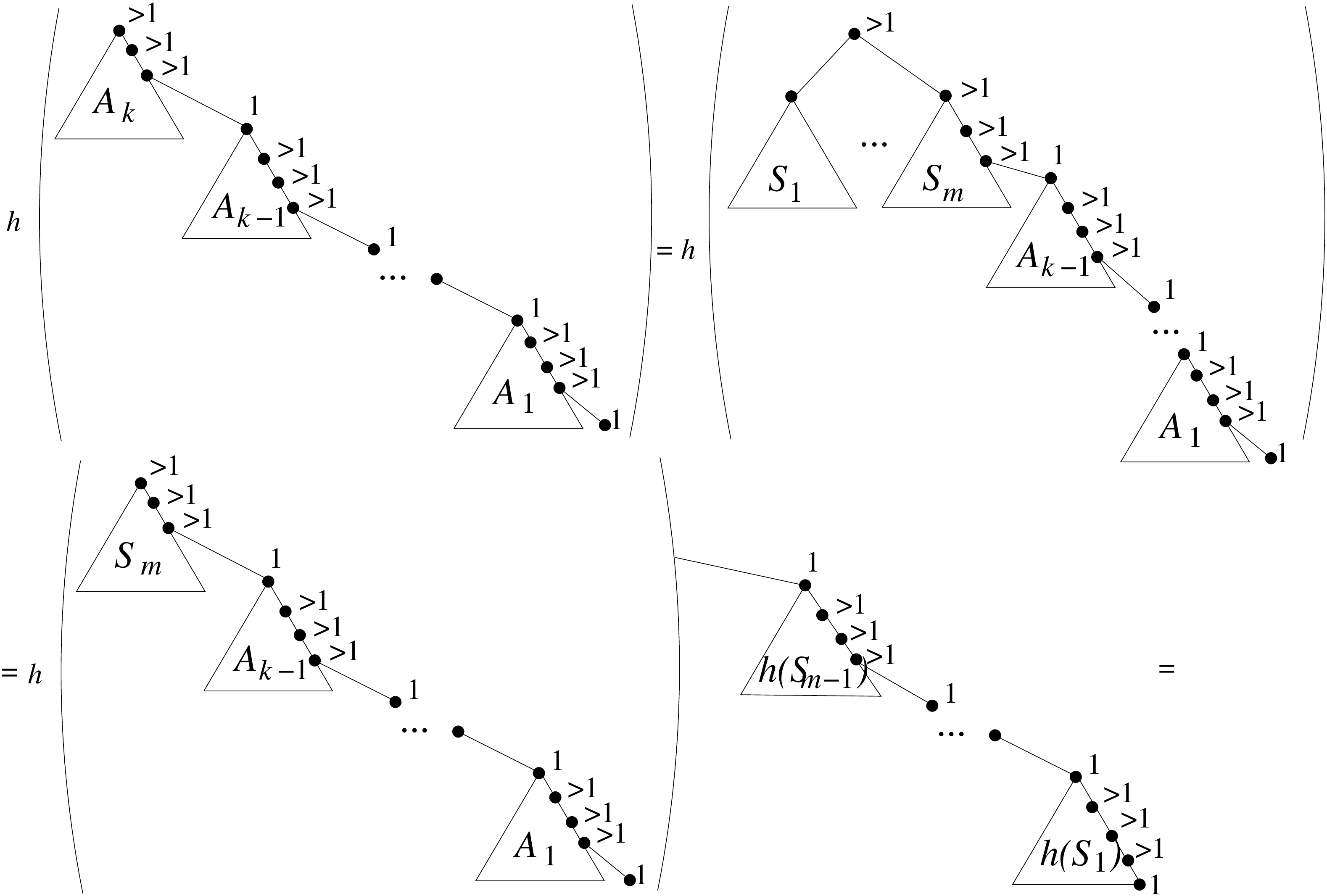}
\caption{Proving the property in Figure \ref{case-to-prove}; continuation is in Figure \ref{case-proof-2}.}\label{case-proof}
\end{center}
\end{figure}

\begin{figure}[!h]
\begin{center}
\includegraphics[scale=0.4]{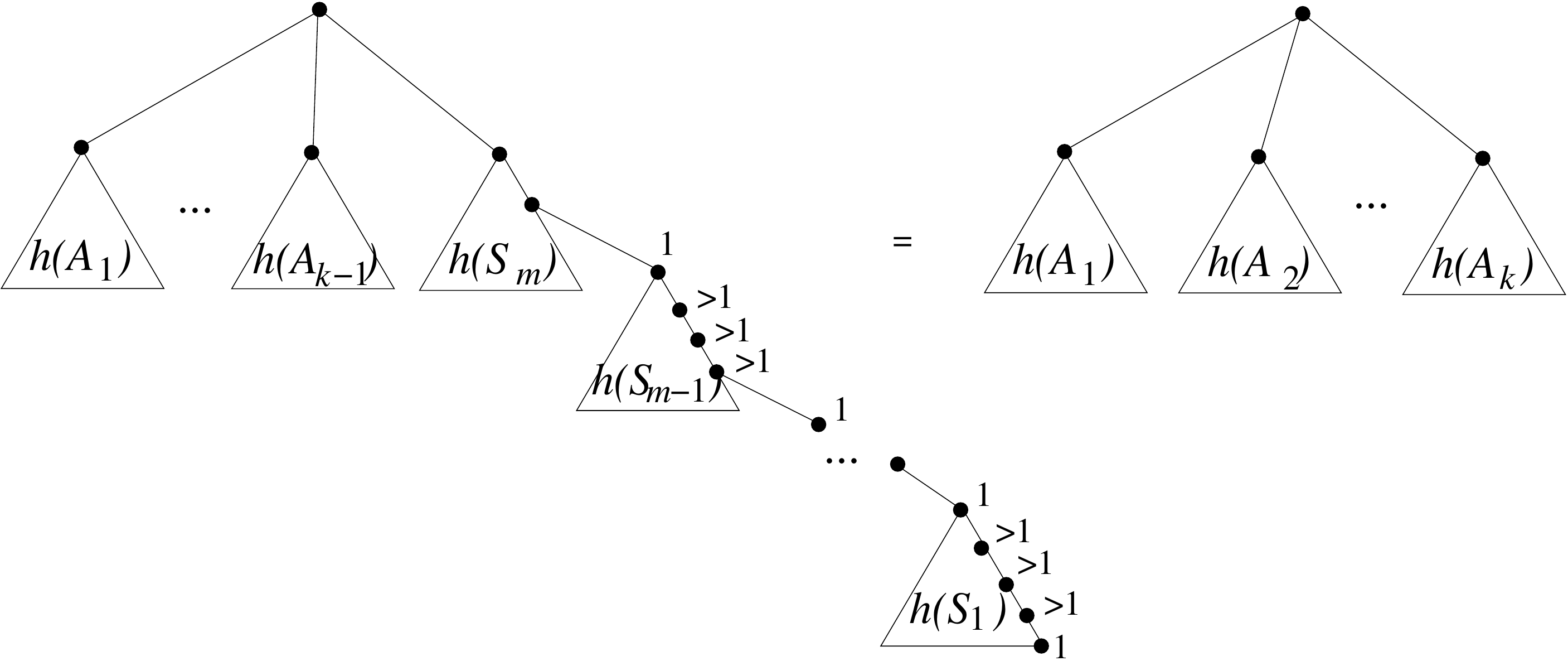}
\caption{Proving the property in Figure \ref{case-to-prove}; continuation of Figure \ref{case-proof}.}\label{case-proof-2}
\end{center}
\end{figure}

\end{proof}

\begin{remark} When defining $h$ in the indecomposable case (see Figure \ref{involutionH}) we can say that the root of the original tree goes to the rightmost leaf of the image tree. Thus, recursively, for each node in a $\beta(1,0)$-tree, there is a corresponding node in $h(T)$.  It can be shown \cite{CKS2} that under $h^2$, each node goes to itself.  \end{remark}

\newpage

\section{Structure of fixed points of $h$}\label{sec-Structure}

All fixed points of the involution $h$ on at most 6 nodes are depicted in 
Figure~\ref{small-fixed}.

\begin{figure}[h]
\begin{center}
\includegraphics[scale=0.5]{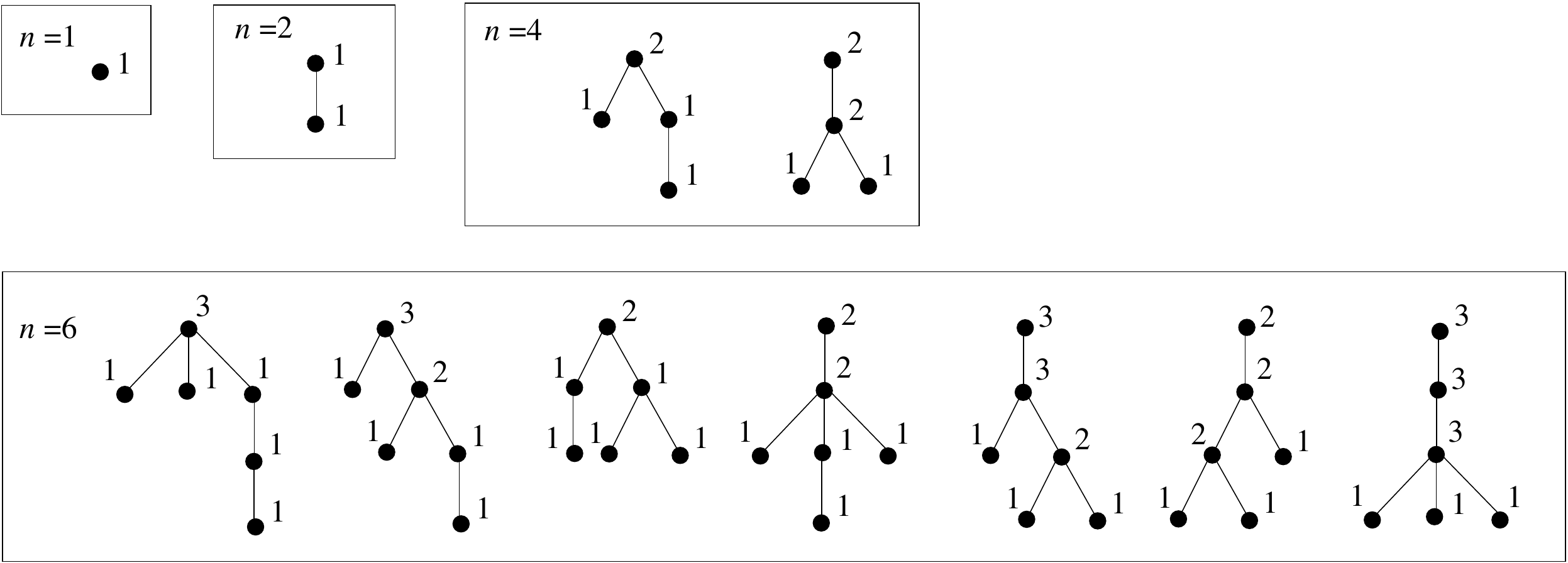}
\caption{All fixed points of $h$ on at most 6 nodes.}\label{small-fixed}
\end{center}
\end{figure}

The number of fixed points for $n=2,4,6,\ldots$ is $1, 2, 7, 30, 143, 728, 3876,\ldots$, which, as we shall see in the next section, is sequence A006013 in OEIS \cite{OEIS}. 

The goal of this section is to prove the following theorem on the structure of the fixed points of $h$ (see Figure~\ref{fixed-structure}).

\begin{theorem}\label{main-structure} If $T$ is a fixed point under $h$, then $T$ has (exactly) one of the following structures:
\begin{itemize}
\item[(F0)] $T$ is a node.
\item[(F1)] $T$ is based on an arbitrary $\beta(1,0)$-tree $A$. 
Change the root of $h(A)$ to $1$ (unless it was $1$ already)  and hang the result from the root of $A$ through a new edge to the right. $T$ is the resulting tree with the root label suitably adjusted.
\item[(F2)] $T$ is based on  a triple $(A_1,A_2,b)$, where $A_1$ is an arbitrary $\beta(1,0)$-tree, $b$ is an integer larger than $1$, and $A_2$ is a fixed point of $h$ with at least two nodes and such that $\mathrm{root}(A_2)\geq b-1$ (and thus, by Theorem \ref{niceProperty}, 
$\mathrm{rpath}(A_2)\geq b-1$). 
 Hang $h(A_1)$ through a new edge to the right from the $(b-1)$-th node on the rightmost path of $A_2$; in the rightmost path of the result, add $1$ to every non-root node from the $(b-1)$-st node upwards (if any) and set the label of the $b$-th node to $1$. Change the root label to $b$ and let $A'_2$ be the result of these operations; finally, hang $A'_2$ from the root of $A_1$ to the right through a new edge, and  rewrite the root as necessary.

%The rightmost child of the root of $T$ has label $b>1$. The rest of the structure is shown schematically on the rightmost picture in Figure \ref{fixed-structure}. There $A_2$ is the $\beta(1,0)$-tree obtained by removing $A_1$, $h(A_1)$ and the two edges in that picture and subtracting $1$ from each node on the rightmost path between the edges; the root of $A_2$ maybe required to be updated accordingly. One has that $A_2$ must be a fixed point. 

\end{itemize}
In particular, except for the one node tree, there are no fixed points of $h$ on an odd number of nodes.
\end{theorem}

The smallest $\beta(1,0)$-tree whose structure is of type F2 is the one to the right for the case $n=4$ in Figure~\ref{small-fixed}. Also, all but the first and third  $\beta(1,0)$-trees for the case $n=6$  in Figure~\ref{small-fixed} have this structure.

%\begin{enumerate}
%\item \textit{From an arbitrary $\beta(1,0)$-tree $T_1$:} let $T'_1$ be the same as $h(T_1)$, but changing the label of the root  to $1$ (if it was already $1$, $T'_1=h(T_1)$; then hang $T'_1$ from the root of $T_1$ through a new edge, so that $T'_1$ becomes the rightmost subtree; finally, change the root as necessary to make the result a proper $\beta(1,0)$-tree.
%\item \textit{From a triple $(T_1,T_2,b)$, where $T_1$ is an arbitrary $\beta(1,0)$-tree, $T_2$ is a fixed point of h, $b$ is an integer larger than $1$, and the root of $T_1$ is at least $b-1$:}  construct $T'_1$ as in the previous case; hang $T'_1$ through a new edge from the $(b-1)$-st vertex on the right path of $T_2$; add $1$ to every non-root vertex on this right path, from the $(b-1)$-st upwards (if any); change the root to $b$ and let $T'_2$ be the result of these operations; finally, hang $T'_2$ from the root of $T_1$ through a new edge, so that it becomes the rightmost subtree, and change to root of $T_1$ as necessary.
%\end{enumerate}

\begin{figure}[h]
\begin{center}
\includegraphics[scale=0.5]{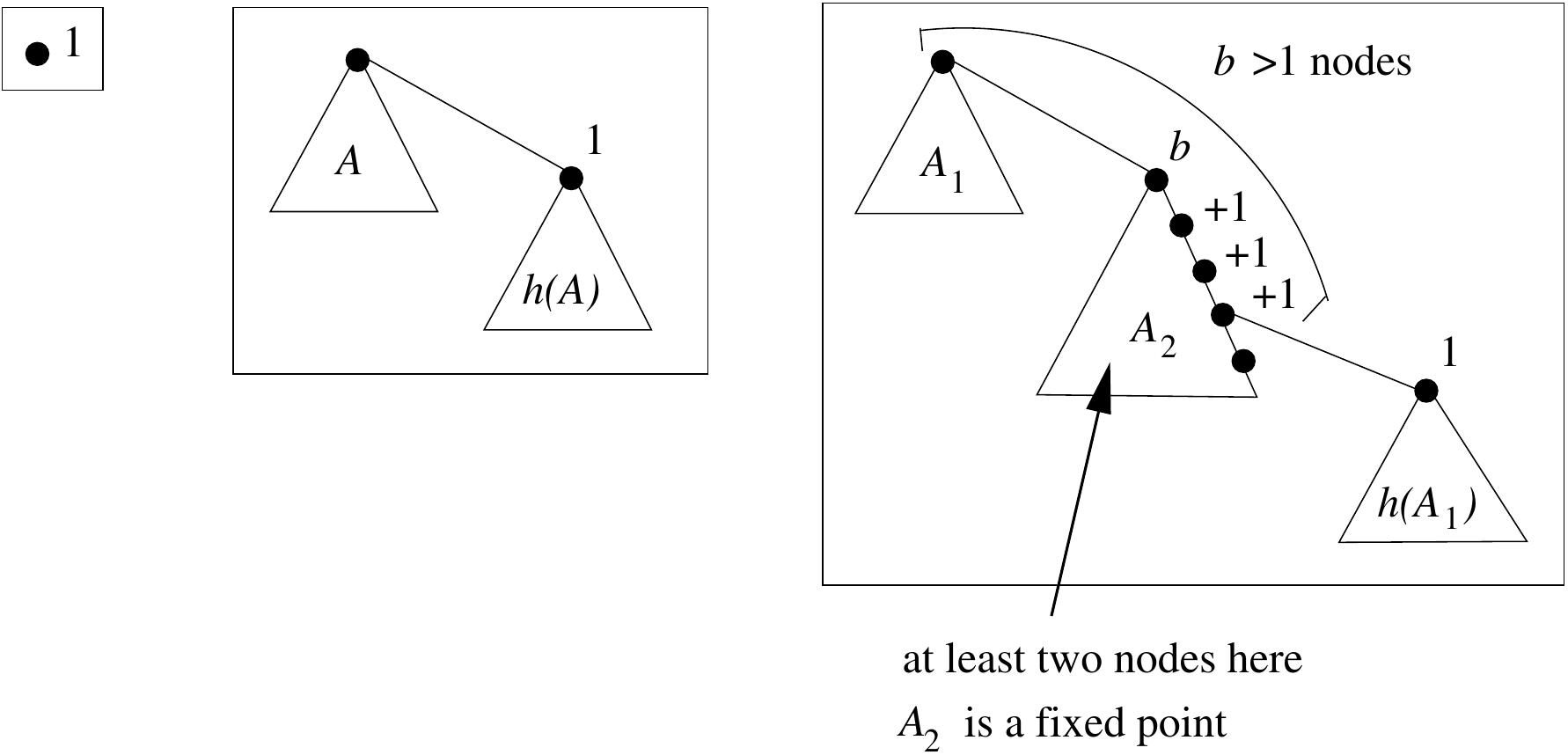}
\caption{All possible structures of fixed points of $h$.}\label{fixed-structure}
\end{center}
\end{figure}

The remaining of the section is devoted to proving some lemmas that will imply Theorem~\ref{main-structure}.

\begin{lemma}\label{lemma-1} If in a fixed point $T$ of $h$ the root's label equals $1$, then $T$ is either a node or an edge. \end{lemma}

\begin{proof} Let $h(T)=T$ and suppose that $T$ is not a node or an edge. Then it is easy to see that either the root has more than one child (that is, $\sub(T)>1$), in which case $\roots(T)>1$, or the length of the rightmost path is more than 1 (that is, $\rpath(T)>1$), in which case, by Theorem~\ref{niceProperty}, $\roots(T)>1$.\end{proof}

\begin{lemma}\label{lemma-2}  
Let  $T$ be a tree with structure as described by one of the items F0, F1 or F2 in Theorem~\ref{main-structure}. Then $T$ is a fixed point of $h$.
%The structures in Figure \ref{fixed-structure} are fixed points of $h$.
\end{lemma}

\begin{proof} It follows from the definition of $h$ that the one node $\beta(1,0)$-tree is a fixed point. Moreover, from the definition of $h$  and the fact that $h^2(A)=A$, we see that trees having structure F1 are also  fixed points. To prove that a tree having structure F2 (recall the right hand-side of Figure \ref{fixed-structure}) is a fixed point of $h$, we compute its image under $h$ by applying first the decomposable case in the definition and then the indecomposable case; the result is shown schematically in Figure~\ref{fixed-structure-proof-1}. 

\begin{figure}[h]
\begin{center}
\includegraphics[scale=0.5]{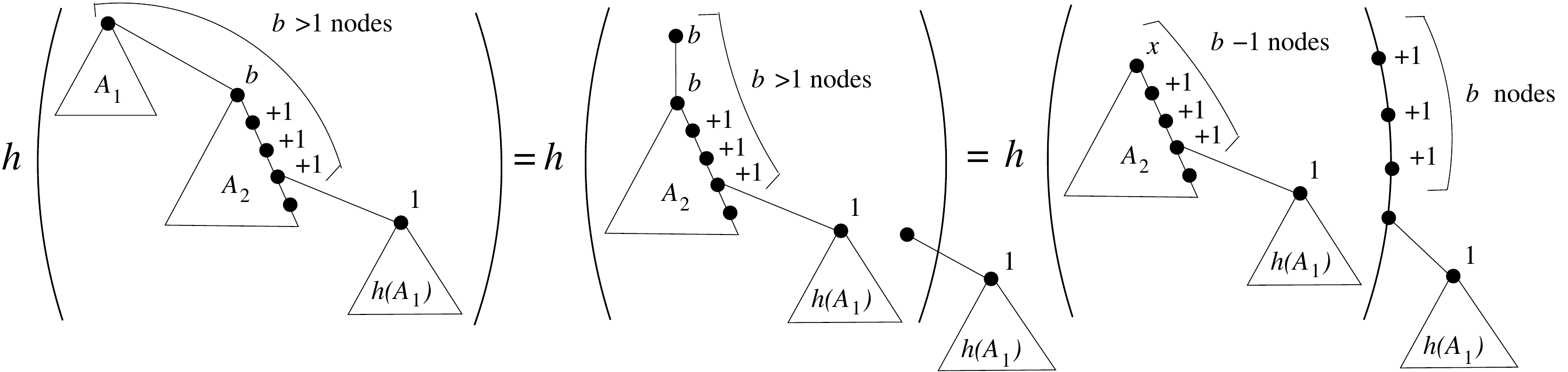}
\caption{Initial steps to prove that the rightmost structure in Figure \ref{fixed-structure}  is a fixed point of $h$.}\label{fixed-structure-proof-1}
\end{center}
\end{figure}

Thus, it is enough to prove  the property of $h$ described by  Figure~\ref{need-to-prove}, where $A_2$ is a fixed point. As $h$ is an involution, it is enough to show that the image of the tree on the right hand-side is the tree 
to which $h$ is applied on the left hand-side. As this is immediate from the definition of $h$, the lemma is proved.

%This can be proved by applying $h$ to both parts of the equality in 
%Figure \ref{need-to-prove} and using $h^2(T)=T$ for any $\beta(1,0)$-tree $T$,
% which is shown by the first equality in Figure \ref{fixed-structure-proof-2}.
% The second equality in Figure \ref{fixed-structure-proof-2} is obtained by 
%applying the definition of $h$ in Figure~\ref{involutionH} (the decomposable 
%case). Figure \ref{fixed-structure-proof-2}  shows that the property of $h$ 
%in Figure \ref{need-to-prove} holds, which completes our proof of the lemma.

\begin{figure}[h]
\begin{center}
\includegraphics[scale=0.5]{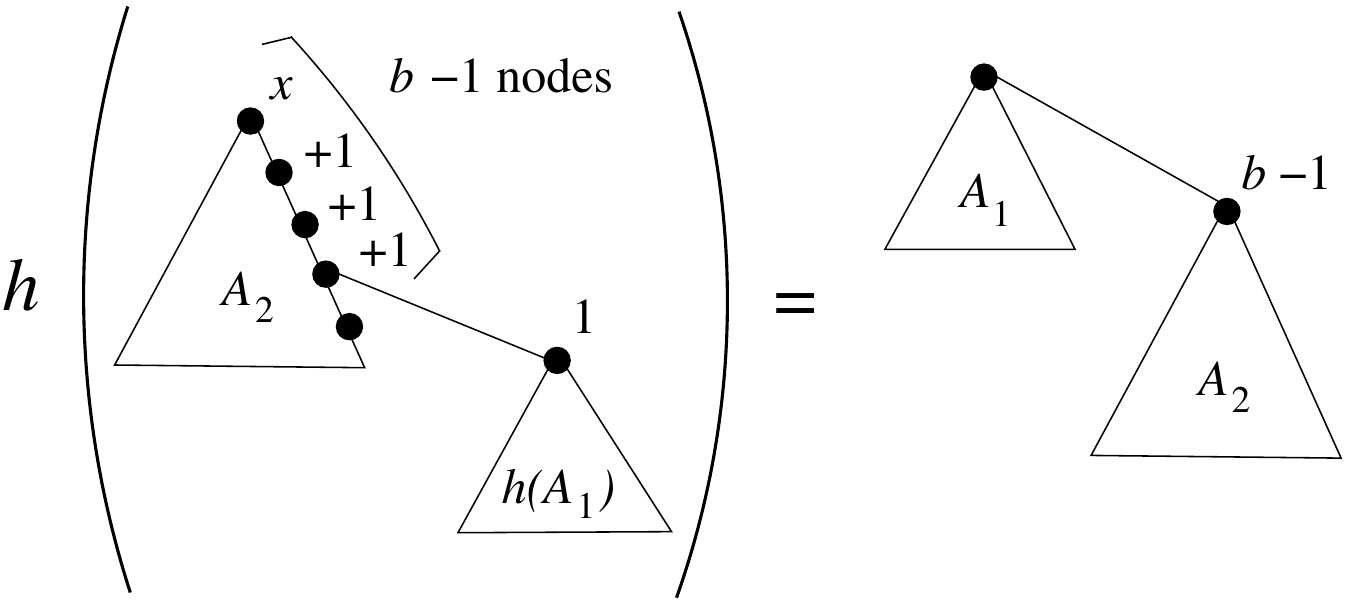}
\caption{A property of $h$ to be proved. Here $A_2$ is a fixed point.}\label{need-to-prove}
\end{center}
\end{figure}

%\begin{figure}[h]
%\begin{center}
%\includegraphics[scale=0.5]{fixed-structure-proof-2.pdf}
%\caption{Proving the property of $h$ in Figure \ref{need-to-prove}.}\label{fixed-structure-proof-2}
%\end{center}
%\end{figure}

\end{proof}

\begin{lemma}\label{lemma-3}  If $T$ is a fixed point of $h$ then $T$ has one of the structures F0, F1 or F2 in Theorem~\ref{main-structure}. \end{lemma}

\begin{proof} 

Suppose $T$ is a fixed point of $h$. If $\roots(T)=1$ then by Lemma \ref{lemma-1} $T$ is either a node or an edge. We assume that $\roots(T)>1$.  Then the structure of $T$ must be one of the two structures in Figure \ref{two-structures}, where $A_1$, $A_2$ and $A_3$ can be single node trees. 

\begin{figure}[h]
\begin{center}
\includegraphics[scale=0.5]{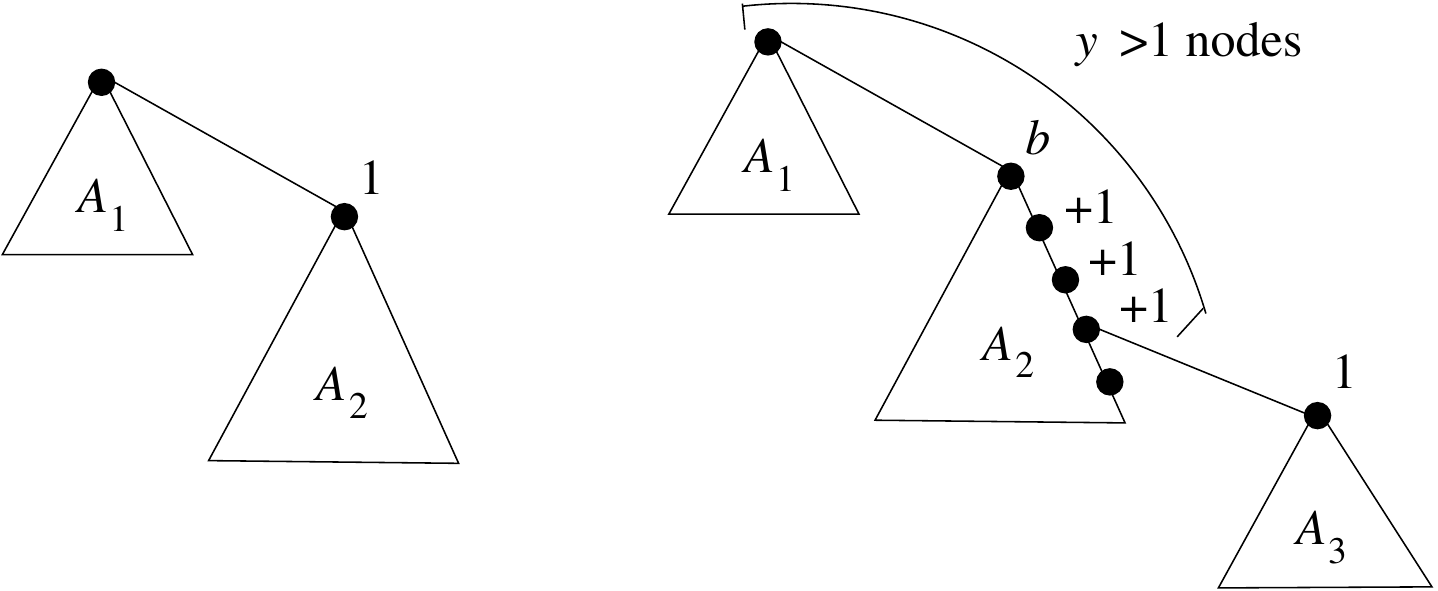}
\caption{A $\beta(1,0)$-tree $T$ with $\roots(T)>1$ has one of these two structures.}\label{two-structures}
\end{center}
\end{figure}

If the structure of $T$ is as the one on the left in Figure~\ref{two-structures}, by applying $h$  (see Figure~\ref{easy-case}) it becomes clear that for $T$ to be a fixed point we must have $A_2=h(A_1)$ and $A_1=h(A_2)$ (which are in fact equivalent conditions since $h$ is an involution).  Thus, the structure of $T$ is as given by F1.

\begin{figure}[h]
\begin{center}
\includegraphics[scale=0.5]{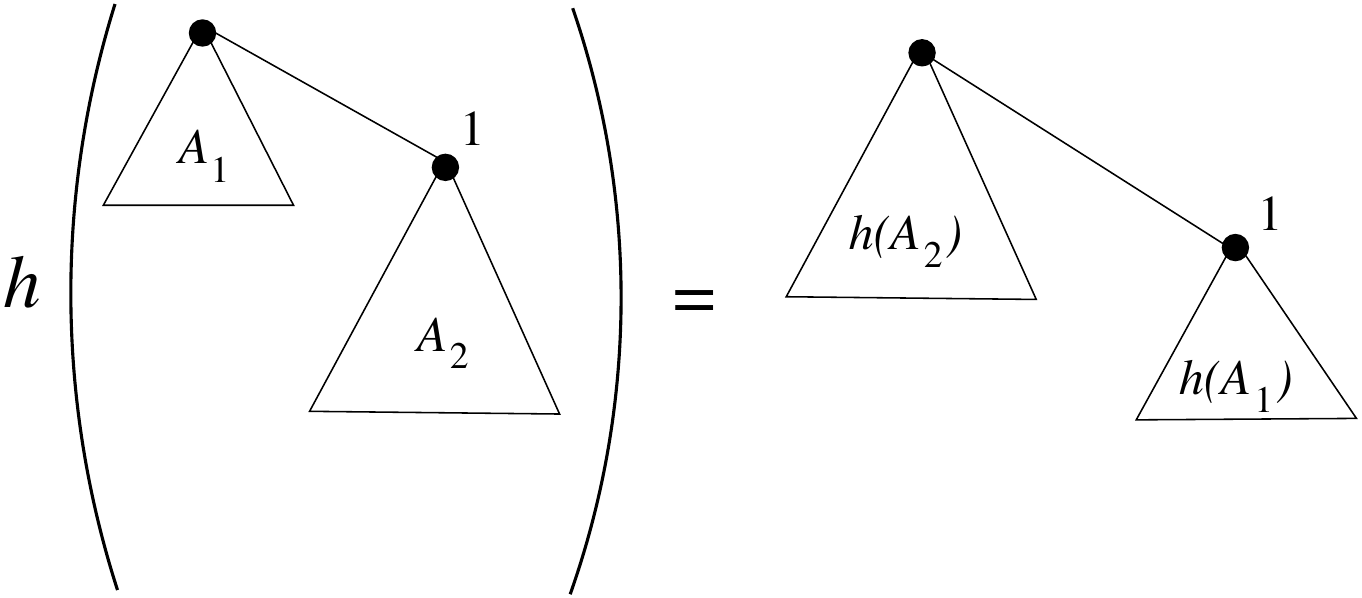}
\caption{Applying $h$ to the structure on the left of  Figure \ref{two-structures}.}\label{easy-case}
\end{center}
\end{figure}

Finally, suppose that the structure of $T$ is as that on the right of Figure~\ref{two-structures}.
We begin by applying $h$ to $T$, which is shown schematically in Figure~\ref{structure-proof-1}. Comparing $T$ with its image under $h$ we conclude that $A_3=h(A_1)$, $y=b$, and that the first equality in Figure~\ref{structure-proof-2} must hold. The second equality in Figure~\ref{structure-proof-2}  is easy to check applying the definition of $h$. Comparing the first and the last trees in this figure  we conclude that $A_2$ must be a fixed point under $h$. Thus, the structure of $T$ is as given by F2.

\begin{figure}[h]
\begin{center}
\includegraphics[scale=0.5]{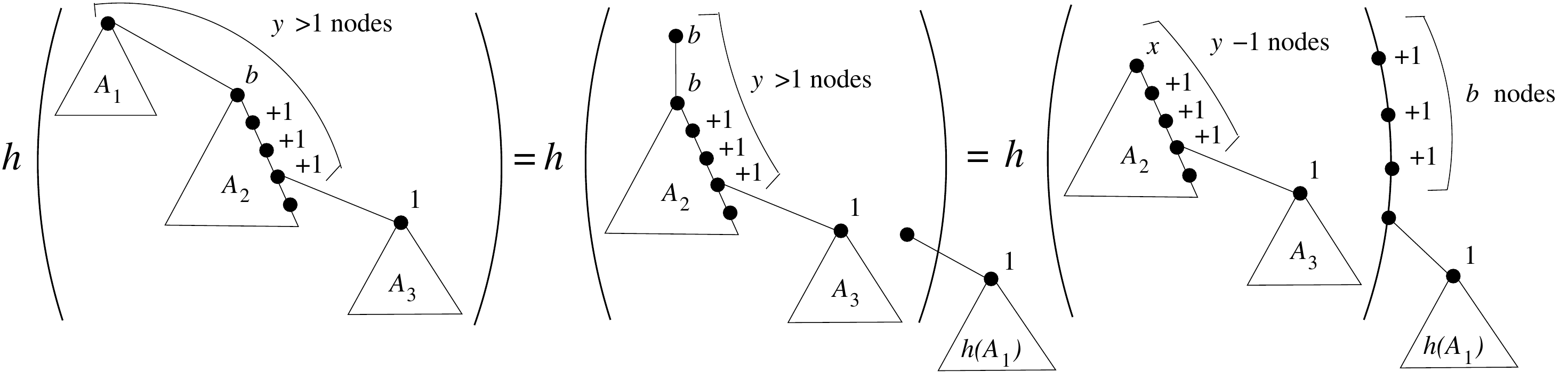}
\caption{Applying $h$ to the right structure in Figure \ref{two-structures}.}\label{structure-proof-1}
\end{center}
\end{figure}

\begin{figure}[h]
\begin{center}
\includegraphics[scale=0.5]{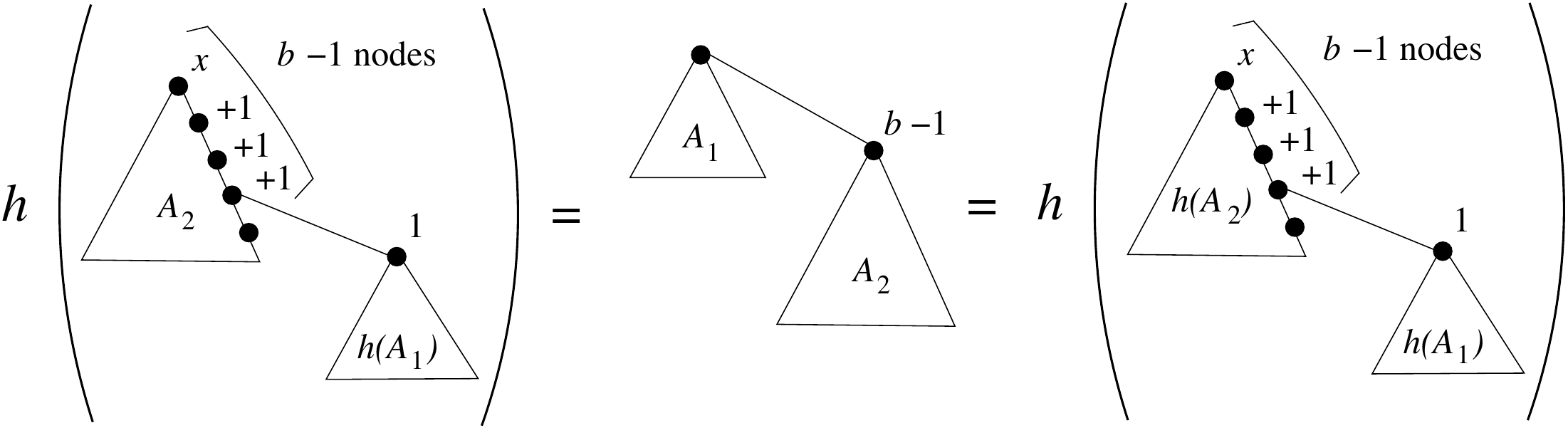}
\caption{A property of the structure under consideration coming from Figure \ref{structure-proof-1} .}\label{structure-proof-2}
\end{center}
\end{figure}

\end{proof}

A proof of Theorem \ref{main-structure} is now given by Lemmas \ref{lemma-2} and \ref{lemma-3}. That the number of nodes of a fixed point must be even follows immediately by induction.

\section{Enumeration of fixed points on $h$}\label{sec-Enum}

In this section we use the structure of fixed points given in Theorem~\ref{main-structure} to prove that $h$ has $\frac{1}{n}\binom{3n-2}{n-1}$ fixed points with $2n$ nodes. (Recall that the only fixed point with an odd number of nodes is the one node tree.)

For $n\geq 1$, let $a_n$   be the number of fixed points of $h$ with $2n$ nodes, and let $a_{n,k}$ be the number of those with root label equal to $k$. Let $b_n$ be the number of $\beta(1,0)$-trees with $n$ nodes, and let $b_{n,k}$ be the number of those with root label equal to $k$, except that, for technical convenience, we take $b_{1,0}=1$ and $b_{1,1}=0$. The corresponding generating functions are denoted by $A(x),A(x,y),B(x)$, and $B(x,y)$.

As mentioned in the introduction, there is a bijective correspondence between $\beta(1,0)$-trees and rooted non-separable maps. Under this bijection (described in detail in Section~\ref{sec-RNPM}), trees with $n$ nodes are mapped to maps with $n$ edges, and, moreover, the label of the root of the tree plus one corresponds to the degree of the root face of the map.
Thanks to this bijection, we can use expressions for $B(x)$ and $B(x,y)$ that were found in the map enumeration context. Brown~\cite{nonsepB} gave the following parametric expression for the series $\tilde{B}(x)=B(x)-x$:
$$\tilde{B}(x)= u(x)^2(1-2u(x)),$$  
where  $x=u(x)(1-u(x))^2.$
A simple application of Lagrange's inversion formula gives
$$u(x)=\sum_{n\geq 1} \frac{1}{n}\binom{3n-2}{n-1} x^n=x+2x^2+7x^3+30x^4+143x^5+\cdots$$
(see sequence A006013 on OEIS).

Brown also gave the following equation relating $\tilde{B}(x)$ and $\tilde{B}(x,y)=(B(x,y)-x)y$:
\begin{equation}
\label{eq:bxy}
\tilde{B}(x,y)^2+[1-y+xy^2-y\tilde{B}(x)]\tilde{B}(x,y)-xy^2(\tilde{B}(x)+x(1-y))=0.
\end{equation}

Theorem~\ref{main-structure} allows us to find an equation linking the series $A(x,y)$ and $B(x,y)$.

\begin{lemma} The series $A(x,y)$ and $B(x,y)$ are related by the equation
\begin{equation}
A(x,y)=yB(x,y)+y^2B(x,y)\frac{A(x,y)-A(x,1)}{y-1}.
\label{eq:AB}
\end{equation}
\end{lemma}

\begin{proof}
The first summand clearly corresponds to trees having structure F1 in Theorem~\ref{main-structure} (here is where setting $b_{1,0}=1$ and $b_{1,1}=0$ makes the formula slightly more compact). Trees having structure F2 give rise to the second summand as follows. Given an arbitrary tree $A_1$ with $n_1$ nodes and root $k_1$ and a fixed point $A_2$ with $2n_2$ nodes and root $k_2$, the construction gives $k_2$ fixed points, all of them with $2n_1+2n_2$ nodes and with roots equal to $k_1+2,k_1+3,\ldots, k_1+k_2+1$. The formula in the statement follows by observing that
$$y^2\frac{A(x,y)-A(x,1)}{y-1}=\sum_{n,k} a_{n,k}x^ny^2\frac{y^k-1}{y-1}  =\sum_{n,k} a_{n,k} x^n(y^2+\cdots+y^{k+1}).$$
\end{proof}

We next solve equation~(\ref{eq:AB}) and show that actually $A(x)=u(x)$.

\begin{theorem}\label{mainEnum}
The series $A(x)$ satisfies the equation $A(x)(1-A(x))^2=x$ and thus the number of fixed points of $h$ is $a_n=\frac{1}{n}\binom{3n-2}{n-1}$. Moreover, the series $A(x,y)$ satisfies the equation
$$[2yA(x)-1]A(x,y)^2-[3yA(x)^2-3yA(x)+1]A(x,y)+yA(x)^3-2yA(x)^2+A(x)y=0.$$
\end{theorem}

\begin{proof}
We use the kernel method to obtain $A(x)=A(x,1)$. We rewrite equation~(\ref{eq:AB}) as
\begin{equation}
A(x,y)(y-1-y^2B(x,y))=y^2B(x,y)(1-A(x))-yB(x,y).
\label{eq:kernel}
\end{equation}

Suppose that  $y(x)$ is a power series such that $y(x)-1-y(x)^2B(x,y(x))=0$. By substituting into~(\ref{eq:kernel}), we obtain
$$A(x)=\frac{y(x)-1}{y(x)},$$
so it only remains to find $y(x)$. In terms of $\tilde{B}(x,y)$, $y(x)$ satisfies $$y(x)-1-xy(x)^2-y(x)\tilde{B}(x,y(x))=0.$$
The resultant of this equation and equation~(\ref{eq:bxy}) is
$$[1-y(x)][-2xy(x)^3+y(x)^2(\tilde{B}(x)+2x+1)-2y(x)+1],$$
so $y(x)$ is clearly a root of the second factor. Writing $x=u(x)(1-u(x))^2$ and $\tilde{B}(x)=u(x)^2(1-2u(x))$, we obtain that $y(x)$ is a root of
$$[1-2u(x)y(x)][1-y(x)+u(x)y(x)]^2.$$
As $y(x)$ must be a power series, it is a root of the second factor and hence
$y(x)=(1-u(x))^{-1}$. From this it follows immediately that $A(x)=u(x)$, as claimed.

The equation for $A(x,y)$ follows by eliminating $B(x,y)$ from equations~(\ref{eq:kernel}) and~(\ref{eq:bxy}), and then writing $B(x)$ and $x$ in terms of $A(x)$.
\end{proof}

The first  few coefficients of $A(x,y)$ are
$$A(x,y)=xy+2x^2y^2+(3y^2+4y^3)x^3+(9y^2+13y^3+8y^4)x^4+\cdots$$

We remark that $A(x)$ is closely related to the generating function $T(x)$ for ternary trees by number of internal nodes (or also, among others, non-crossing trees by number of edges). It is well-known that $T(x)$ satisfies $T(x)=1+xT(x)^3$ (see sequence  A001764 in OEIS). Then it is easy to check, by computing a resultant or otherwise, that $A(x)=xT(x)^2$.

\section{A link to fixed points of taking the dual map on rooted non-separable planar maps}\label{sec-RNPM}

\begin{definition} A {\em planar map} is a connected graph embedded in the sphere with no edge-crossings. Such embeddings are considered
up to continuous deformation and multiple edges and loops are allowed.  A map has {\em vertices} (points), {\em edges}, and {\em faces} (disjoint simply connected domains).
%The {\em outer-face} is unbounded, the {\em inner-faces} are bounded. 
\end{definition}

The maps we are dealing with were considered by Tutte~\cite[Ch. 10]{Tutte1998}, who founded the
enumeration theory of planar maps in a series of papers in the 1960s.

\begin{definition} A {\em cut vertex} in a map is a vertex whose deletion disconnects the
map. A map is {\em non-separable} if it has no loops and no cut vertices. \end{definition}

The maps considered by us are {\em
rooted}, meaning that a directed edge $vw$ is distinguished. The face that lies to the right of the root edge when going from $v$ to $w$ is called the \emph{root-face} (in the figures it is customary to make it agree with the unbounded face). The vertex $u$ is called the \emph{root-vertex}.

%Without loss of generality, we can assume that the root is incident to the 
%outer-face, and the outer-face lies on its right side while following the 
%root orientation. For such an orientation, the outer-face will be the
% {\em root-face}. In general, the root face of a planar map is the face 
%adjacent to the root that lies to the right of it while following the root 
%direction. Also, the vertex from which the root comes out is called the
% {\em root-vertex}.

All rooted non-separable planar maps on 4 edges are given in Figure~\ref{maps}.

\begin{figure}[h]
\begin{center}
\includegraphics[scale=0.4]{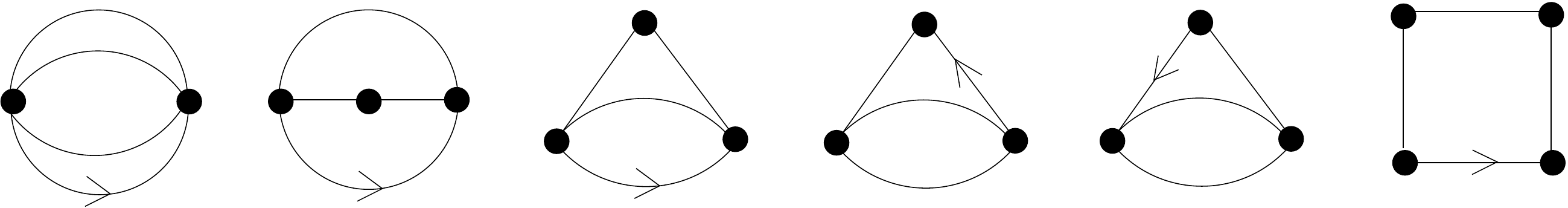}
\caption{All rooted non-separable planar maps on 4 edges.}\label{maps}
\end{center}
\end{figure}

\begin{definition} Two rooted maps are {\em isomorphic} if there is a homeomorphism of the sphere taking one map into the other, preserving incidences between vertices, edges and faces, and preserving the root-vertex, root-edge and root-face.\end{definition}

If $M$ is a rooted map, we define the dual map $M^*$ as follows. As a plane graph, $M^*$ is the dual plane graph of $M$. If $e=vw$ is the root-edge of $M$, then the root edge of $M^*$ is $xy$, where $x$ corresponds to the root-face of $M$, and $xy$ is defined as follows. Let $e^* = xz$ be the edge of $M^*$ crossed by $e$. Then take as $xy$ the edge following $xz$ in counter-clockwise order. Notice that in this way, the root vertex and face of $M^*$ correspond, respectively, to the root face and vertex of $M$. See Figure \ref{dual} for an example, where vertices of $M^*$ are white and edges are dashed. It is easy to check that with this definition, duality is an involution on rooted maps, that is, $M^{**}$ and $M$ are isomorphic as rooted maps. 

\begin{figure}[h]
\begin{center}
\includegraphics[scale=0.5]{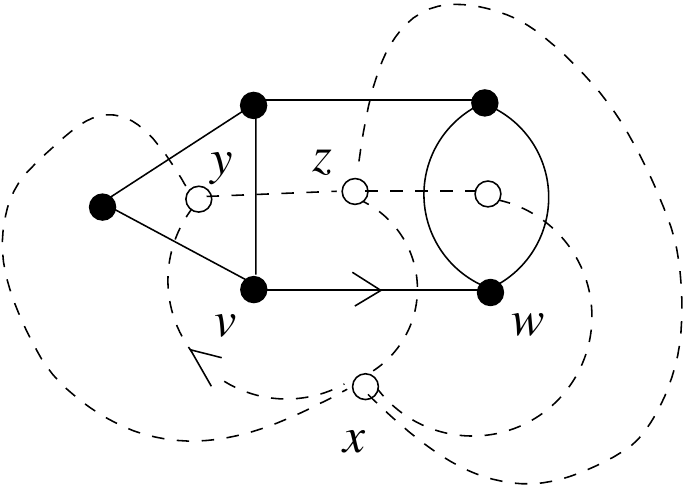}
\caption{Taking the dual map operation.}\label{dual}
\end{center}
\end{figure}

\begin{definition} A rooted map $M$ is self-dual if $M$ and $M^*$ are isomorphic. \end{definition}

Self-dual maps for three classes of planar maps were enumerated in \cite{KMN}. In particular, it was shown there that the number of self-dual rooted non-separable planar maps is given by the same formula as the number of fixed points of $h$ (see Theorem \ref{mainEnum}). For $4$ edges, the two self-dual maps are the two in the middle of Figure~\ref{maps}.

There is a natural (known) bijective map from $\beta(1,0)$-trees to rooted non-separable planar maps that we call  {\em standard} because it naturally preserves the structure of the objects involved. The map can be described as follows. Given a $\beta(1,0)$-tree, begin by assigning to each of its leaves the rooted map with one edge with the root-vertex labeled $R$ (for root-node) and the other vertex labeled $*$ (auxiliary symbol), as shown in  Figure~\ref{maps-trees}. Assume that, recursively, each child of a node $x$ in a $\beta(1,0)$-tree is assigned a map with the root-vertex labeled $R$ and the auxiliary symbol $*$ labeling a non-root node on the root-face.  To produce the map corresponding to $x$, glue the maps corresponding to its children from left to right so that the $*$ node in the first map is glued with the $R$ node in the second map; the $*$ node in the second map is glued with the $R$ node in the third map; etc (if $x$ has a single child, we do not make any gluing). Then remove the orientations from ``old" root-edge(s) and add a new root-edge from the rightmost $*$ node to the leftmost $R$ node; change the label of the rightmost $*$ to be $R$ (all other $*$s and $R$s are removed).  Finally, if the label of $x$ was $a$, label by $*$ the $a$-th node on the root-face counted from  $R$ in counter-clockwise direction. See Figure \ref{maps-trees} for an example.

\begin{figure}[h]
\begin{center}
\includegraphics[scale=0.5]{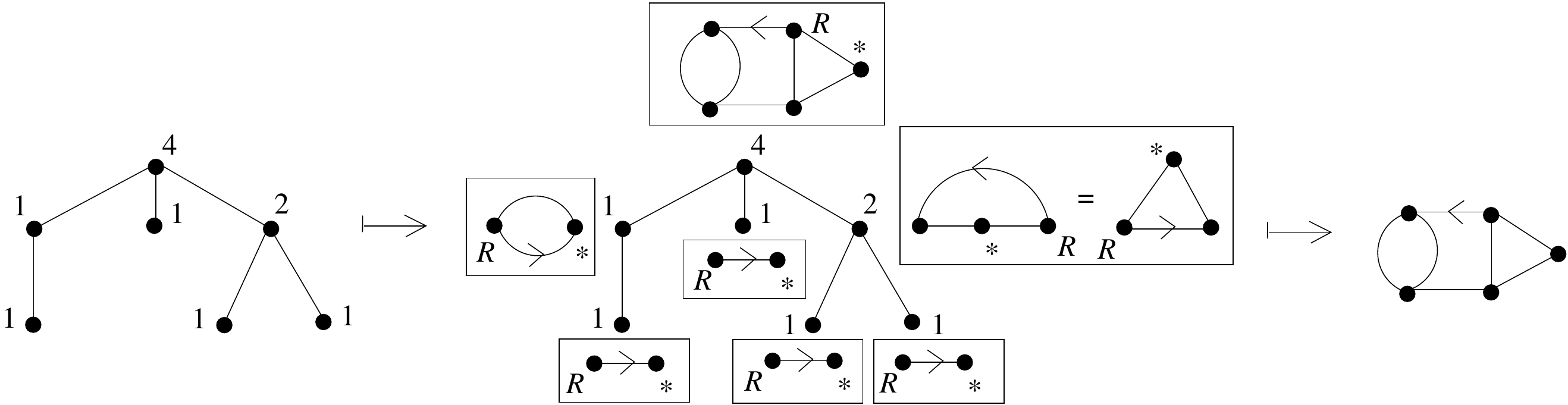}
\caption{Mapping bijectively a $\beta(1,0)$-tree to a rooted non-separable planar map.}\label{maps-trees}
\end{center}
\end{figure}

Though being equinumerous, fixed points under $h$ unfortunately do not go to fixed points under taking the dual map on rooted non-separable planar maps when applying the standard bijection. This raises the following open problem.

\begin{problem} Find a combinatorial (bijective) explanation of the fact that the number of fixed points under $h$ on $\beta(1,0)$-trees is equal to the number of fixed points under taking the dual map on rooted non-separable planar maps.\end{problem}

However, one can restrict him/herself to the standard bijection and raise the following questions.

\begin{problem} Describe the image of fixed points of $h$ under applying the standard bijection. \end{problem}

\begin{problem} Describe the image of non-separable self-dual maps under applying the reverse of the standard bijection.\end{problem}

Solving the last two problems will bring new classes of objects equinumerous with fixed points of $h$.

\section{More open bijective questions}\label{more}

%The fact that the number of fixed points of $h$ is given by the sequence A006013 in OEIS \cite{OEIS} raises several bijective questions, one of which is given below.
%
%\begin{problem} Find a bijection between fixed points of $h$ on $2n$ nodes and the number of ways to connect $2n-2$ points labeled $1,2,\ldots,2n-2$ in a line with $0$ or more noncrossing arcs above the line such that each maximal contiguous sequence of isolated points has even length. For example, with arcs separated by dashes, for $n=3$, $7$ counts no arcs, $12$, $14$, $23$, $34$, $12-34$, and $14-23$. It does not count $13$ because $2$ is an isolated point. 
%{\bf Actually there seem to be something wrong with this combinatorial object; first of all, it looks like to accept $23$ as a valid object for 4 points, we require runs of isolated points to be of even length {\em cyclicly}, but even then, I was only able to construct $26$ objects corresponding to $n=4$, not 30. I hope I simply don't understand the definition, though I doubt it ...}\\
%{\bf AdM: I also get 26.}
%\end{problem}
%
%
Besides taking the dual of a non-separable map, there are several other involutions on combinatorial objects whose fixed points are known to be equinumerous with fixed points of $h$. We first review some results from~\cite{DFN}, where three classes of trees and a class of polyominoes are counted under reflection, and later we point at another connection with non-separable maps. 

Let us begin by recalling some definitions. A rooted plane (unlabeled) tree is \emph{ternary} if every node has outdegree equal to 0 or to 3, and it is \emph{even} if every node has even outdegree. A \emph{non-crossing tree} is a tree on the vertices of a convex regular polygon whose edges do not cross; moreover, a vertex of the polygon is distinguished as the root of the tree. The \emph{reflection} of a rooted plane tree is defined by recursively interchanging the order of the children at every node, whereas  for a non-crossing tree it consists in taking the image under the reflection by a bisector of the polygon through the root. 

A directed polyomino is \emph{diagonally convex} if all cells whose centers are on a line of slope $-1$ form a continous chain. The \emph{reflection} of such a polyomino is taken with respect to the line of slope $1$ through the center of the left-bottom cell. 

Ternary trees with $n$ internal nodes, even trees with $2n$ edges, non-crossing trees with $n$ edges, and diagonally convex directed polyominoes with $n$ diagonals are all  enumerated by $1/(2n+1)\binom{3n}{n}$ (see sequence A001764 in OEIS). 

An object from one of these four families that is equal to its reflection is called \emph{symmetric} (Figure~\ref{fig:symmetric} shows the case $n=3$).  Recall that $a_n$ denotes the number of fixed points of $h$ with $2n$ nodes. It  was proved in~\cite[Theorem 1]{DFN} that for odd $n$, $a_{(n+1)/2}$  is the number of
\begin{itemize}
\item symmetric ternary trees with $n$ internal nodes;
\item symmetric even trees with $2n$ edges;
\item symmetric non-crossing trees with $n$ edges;
\item symmetric diagonally convex directed polyominoes with $n$ diagonals.
\end{itemize}
For even values of $n$ the number of symmetric objects also agrees in the four cases and it equals $1/(n+1)\binom{3n/2}{n/2}$ (so, again, sequence A001764).

\begin{figure}[h]
\begin{center}
\includegraphics[height=4cm]{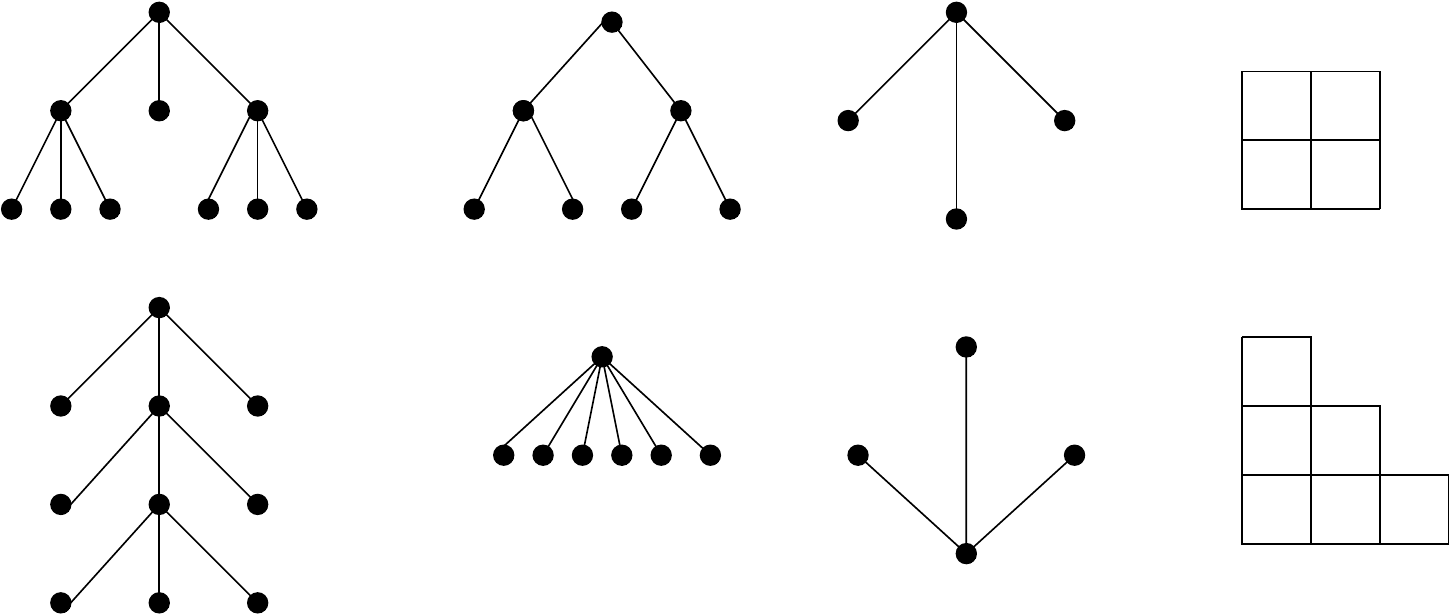}
\caption{Examples of symmetric ternary, even, and non-crossing trees, and diagonally convex directed polyominoes.}
\label{fig:symmetric}
\end{center}
\end{figure}

We conclude by mentioning another class of maps equinumerous with fixed points of $h$. In~\cite{nonsepB} (see formula 8.21), it is shown that $a_n$ also counts fixed points under a $\pi$-degree rotation of a  subclass of non-separable planar maps. More concretely, let $M$ be a  rooted planar map where the root-face has degree $2$ and the root-edge is $xy$, and let $M'$ be its image under a $\pi$-degree rotation, with the root of $M'$ being the other edge on the root-face of $M$, oriented from $y$ to $x$.  Then $a_n$ is the number of such maps  $M$ with $2n+1$ edges and such that $M$ and $M'$ are isomorphic as rooted maps.

\begin{problem} Explain bijectively (some of) the links between fixed points of $h$ and the structures discussed in this section. \end{problem}

%\textbf{AdM: given that all examples we have of equinumerous objects are of the fixed point kind, we could either put them all in one section or have one section for the map ones (which are indeed more related  to $\beta(1,0)$-trees) and another one for the trees and the polyominoes.}

\section{Acknowledgements}
 The authors would like to thank Marc Noy for helpful discussions related to the paper, and Anders Claesson for providing us data on fixed points of the involution $h$. The second author was supported by the Spanish and Catalan governments under Projects MTM2011-24097 and DGR2009-SGR1040.

\end{document}